\documentclass[12pt]{amsart}

\usepackage{tgpagella}
\usepackage{euler}
\usepackage[T1]{fontenc}
\usepackage{amsmath, amssymb}
\usepackage[hidelinks]{hyperref}
\usepackage[english]{babel}
\usepackage{mathrsfs}
\usepackage{eucal}
\usepackage[all]{xy}
\usepackage{tikz-cd}

\usepackage[
backend=biber,
style=numeric,
sorting=nyt,
]{biblatex}
\newtheorem{main}{Theorem}

\newtheorem{thm}{Theorem}[section]
\newtheorem*{thm*}{Theorem}
\newtheorem{lem}[thm]{Lemma}
\newtheorem{fact}[thm]{Fact}

\newtheorem{prop}[thm]{Proposition}
\newtheorem*{prop*}{Proposition}
\newtheorem{conj}[thm]{Conjecture}
\newtheorem{cor}[thm]{Corollary}
\newtheorem*{cor*}{Corollary}

\theoremstyle{definition}
\newtheorem{defn}[thm]{Definition}
\newtheorem*{defn*}{Definition}

\newtheorem{remark}[thm]{Remark}

\newtheorem{question}[thm]{Question}

\newtheorem*{question*}{Question}

\newtheorem*{conv*}{Convention}

\def\bb{\mathbb}

\def\bb{\mathbb}

\def\cal{\mathcal}
\def\cM{\mathcal{M}}

\makeatletter

\def\dotminussym#1#2{%
  \setbox0=\hbox{$\m@th#1-$}%
  \kern.5\wd0%
  \hbox to 0pt{\hss\hbox{$\m@th#1-$}\hss}%
  \raise.6\ht0\hbox to 0pt{\hss$\m@th#1.$\hss}%
  \kern.5\wd0}

\newcommand\cU{{\cal U}}
\newcommand{\cN}{{\cal N}}

\newcommand\bN{{\mathbb N}}
\newcommand\bR{{\mathbb R}}
\newcommand\bC{{\mathbb C}}
\newcommand{\F}{\mathbb F}

\def \Th{\operatorname{Th}}
\def \cR{\mathcal R}
\def \cU{\mathcal U}
\def \C{\mathbb C}

\newcommand{\fo}{\text{fo}}

\title{On the first-order free group factor alternative}
\author[Goldbring and Pi]{Isaac Goldbring and Jennifer Pi}
\address{Department of Mathematics\\University of California, Irvine, 340 Rowland Hall (Bldg.\# 400),
Irvine, CA 92697-3875}
\email{isaac@math.uci.edu, jspi@uci.edu}
\urladdr{http://www.math.uci.edu/~isaac, \, https://sites.uci.edu/jpi314/}

\begin{document}

\begin{abstract}
    We investigate the problem of elementary equivalence of the free group factors, that is, do all free group factors $L(\F_n)$ share a common first-order theory?
    We establish a trichotomy of possibilities for their common first-order fundamental group, as well as several possible avenues for establishing a dichotomy in direct analog to the free group factor alternative of Dykema and R\v adulescu. 
    We also show that the $\forall \exists$-theories of the interpolated free group factors are increasing, and use this to establish that the dichotomy holds on the level of $\forall \exists$-theories.  We conclude with some observations on related problems.
    % observations regarding the possible coincidence of the $\forall\exists$-theories of free group factors and matrix ultraproducts, the version of the main question for reduced group C*-algebras of the free groups, and potential attacks on the question of whether or not taking group von Neumann algebras preserves elementary equivalence.
\end{abstract}
\maketitle

\section{Introduction}

A basic construction in the theory of tracial von Neumann algebras associates to every (countable, discrete) group $\Gamma$ its \textbf{group von Neumann algebra} $L(\Gamma)$.  More precisely, if one lets $\ell^2(\Gamma)$ denote the Hilbert space with orthonormal basis $(u_\gamma)_{\gamma\in \Gamma}$, then the left regular representation $\lambda:\Gamma\to U(\ell^2(\Gamma))$ is the unitary representation of $\Gamma$ defined by $\lambda(\gamma)(u_\eta):=u_{\gamma\eta}$.  $L(\Gamma)$ is then defined to be the von Neumann subalgebra of $\mathcal{B}(\ell^2(\Gamma))$ generated by $\lambda(\Gamma)$.  $L(\Gamma)$ is a tracial von Neumann algebra when equipped with the trace $\tau(x):=\langle xu_1,u_1\rangle$ and is a factor precisely when $\Gamma$ is an ICC group, that is, when all nontrivial conjugacy classes of $\Gamma$ are infinite.

It is well-known that the group von Neumann algebra $L(\Gamma)$ may ``forget'' much of the algebraic information about $\Gamma$.  For example, by a celebrated result of Connes \cite{Connes76}, any two infinite ICC amenable groups generate the same group von Neumann algebra, namely the hyperfinite II$_1$ factor $\mathcal{R}$.  On the other hand, there are situations in which $L(\Gamma)$ ``completely remembers'' $\Gamma$, for example the generalized wreath product groups of \cite{IPV13-Superrigid}.  In fact, a famous conjecture of Connes states that $L(\Gamma)$ completely remembers $\Gamma$ when $\Gamma$ is an ICC group with Kazhdan's property (T) \cite{Connes82}. (For more on Connes' rigidity conjecture, see, for example, the introduction of \cite{IPV13-Superrigid}.)

A long-standing question regarding the free group factors $L(\F_n)$ for $n \geq 2$, originally considered by Murray and von Neumann in \cite{MvN43}, is whether $L(\F_m) \cong L(\F_n)$ for distinct $m,n \geq 2$. 
Dykema and R\v adulescu independently introduced the \textbf{interpolated free group factors} in \cite{Dykema94}, \cite{Rad1995random}.

Using these, they proved a dichotomy for the free group factor isomorphism problem; the following is \cite[Corollary 4.7]{Rad1995random}.

\begin{thm} \label{thm: fgf-alternative}
One of the following two statements must hold.
\begin{enumerate}
       \item $L(\F_r) \cong L(\F_s)$ for all $1 < r \leq s \leq \infty$, and the fundamental group of $L(\F_r)$ is $\bR_+$ for all $1 < r \leq \infty$.
       \item $L(\F_r) \not \cong L(\F_s)$ for all $1 < r < s \leq \infty$, and the fundamental group of $L(\F_r)$ is $\{1\}$ for all $1 < r < \infty$.
   \end{enumerate}
\end{thm} 

In this paper, we investigate the problem of free group factor elementary equivalence:
\begin{question}\label{foeequestion}
    Is $L(\F_m) \equiv L(\F_n)$ for some (or all) distinct $m, n\geq 2$?
\end{question}

There are two natural motivations for the previous problem.  First, in many circumstances, when two II$_1$ factors are not isomorphic, they are not even elementarily equivalent.  (See, for example, \cite[Section 4]{GoldbringHart23}.)   Thus, if all free group factors have the same first-order theory, then this might serve as evidence as to why the free group factor isomorphism problem is so difficult.  Of course, if some free group factors have distinct first-order theories, then this would provide a strong refutation of the free group factor isomorphism problem.

A second motivation for Question \ref{foeequestion} is Sela's famous positive resolution \cite[Theorem 3]{Sela-FGF-EE} of the Tarski problem, which asked if all nonabelian free groups have the same first-order theory.  (For this reason, Thomas Sinclair once called Question \ref{foeequestion} the \textbf{noncommutative Tarski problem}.)  A na\"ive attempt toward resolving Question \ref{foeequestion} might be to somehow ``transfer'' the fact that $\bb F_2\equiv \bb F_3$ holds to prove that $L(\bb F_2)\equiv L(\bb F_3)$ is also true.  However, a quick examination of the construction of the group von Neumann algebra as well as the first-order language used to study tracial von Neumann algebras shows that such a transferral of elementary equivalence is not automatic.  In fact, it is natural to ask whether or not there are examples of ICC groups $\Gamma_1$ and $\Gamma_2$ that are elementarily equivalent but whose group von Neumann algebras $L(\Gamma_1)$ and $L(\Gamma_2)$ are not elementarily equivalent.  The first author recently constructed such an example of a pair of II$_1$ factors \cite{GoldbringNonUnifInnerAmenable}.  We note that, by strengthening the assumption by assuming that $\Gamma_1$ and $\Gamma_2$ are \emph{back-and-forth equivalent}, one can indeed conclude that $L(\Gamma_1)$ and $L(\Gamma_2)$ are elementarily equivalent \cite{GHT2023backandforth}.

Our initial hope was to establish a dichotomy for the first-order free group factor problem analogous to Theorem \ref{thm: fgf-alternative} above.  Although we are currently unable to obtain this dichotomy, we are able to obtain an interesting trichotomy:

\newpage
\begin{main}
Exactly one of the following holds:
\begin{enumerate}
    \item $L(\F_r) \equiv L(\F_s)$ for all $1 < r \leq s < \infty$, and the first-order fundamental group of $L(\F_r)$ is $\bR_+$ for all $1 < r \leq \infty$.
       \item $L(\F_r) \not \equiv L(\F_s)$ for all $1 < r < s < \infty$, and the first-order fundamental group of $L(\F_r)$ is $\{1\}$ for all $1 < r < \infty$.
       \item There is $\alpha\in (1,\infty)$ such that the first-order fundamental group of $L(\F_r)$ is $\alpha^\bb Z$ for all $1<r < \infty$. 
\end{enumerate}
\end{main}

Using this trichotomy, we are able to obtain the desired dichotomy if the following conjecture has a positive solution.

\begin{conj}
    If $\cal U$ is an ultrafilter on $\bb N$ and $(r_k)$ is a sequence in $\bb R_{>1}$ for which $\lim_\cal U r_k=\infty$, then $\prod_{\cal U} L(\bb F_{r_k})\cong L(\bb F_\infty)^\cal U$. 
\end{conj}

Note that this fact holds for sequences $r_k$ whose ultralimit is some $r \in \bb R_{> 1}$ (see Corollary \ref{cor: continuity}). This conjecture is also a question of independent interest, since it establishes that the theories of the interpolated free group factors varies continuously in a way that includes $L(\F_\infty)$.

Separately from the trichotomy, we give a few possible avenues toward establishing variants of the dichotomy using results on the preservation of first-order theory under free products. In particular, we show

\begin{enumerate}
    \item If taking the free product with $L(\F_{2 - \epsilon})$ for $\epsilon \in (0,1)$ preserves elementary equivalence, then we obtain the desired dichotomy.
        %More precisely, if we have two elementarily equivalent tracial von Neumann algebras $\cM \equiv \cN$ and $\mathcal{A}$ is one of the von Neumann algebras above, then $\cM * \mathcal{A} \equiv \cN * \mathcal{A}$.
    \item If taking the free product with $L(\bb Z)$ or $\cR$ preserves elementary equivalence and $L(\F_2) \equiv L(\F_n)$ for some integer $n \neq 2$, then $L(\F_r) \equiv L(\F_s)$ for all $r, s \in (1, \infty)$.
    \item If taking the free product with $M_2(\C)$ preserves elementary equivalence and $L(\F_m) \equiv L(\F_n)$ for some integers $m, n$ of opposite parity, then $L(\F_r) \equiv L(\F_s)$ for all $r, s \in (1, \infty)$.
\end{enumerate}

Further, using Dykema's notions of standard embeddings \cite{Dykema93FreeDim}, we also prove the following.

\begin{main}
    If for some $1 < r < s$ there is an embedding $\alpha: L(\F_r) \rightarrow L(\F_s)$ that is both standard and elementary, then $L(\F_x) \equiv L(\F_y)$ for all $x, y \in \bb{R}_{>1}$.
\end{main}

Finally, we give some evidence that the third item in the trichotomy is unlikely to hold. By investigating existential embeddings between the interpolated free group factors, we establish the following result on the 2-quantifier theory of the family $\{L(\F_r)\}_{r \in \bb R_{>1}}$.

\begin{main}
    If the first-order fundamental group of the free group factors is not trivial, then the $\forall \exists$-theory of all interpolated free group factors is the same.
\end{main}

The remainder of the paper is structured as follows. After presenting a few preliminaries in Section 2, we present the trichotomy in Section 3, along with some possibilities for including $L(\F_\infty)$ into the trichotomy. 
We also consider a number of avenues for establishing the dichotomy by using preservation of theories under taking free products. 
We then turn to investigating existential embeddings between the free group factors in Section 4, and obtain that the $\forall \exists$-theory of $L(\F_r)$ is increasing in $r$. 
Lastly in Section 5, we make observations regarding the possible coincidence of the $\forall\exists$-theories of free group factors and matrix ultraproducts, the version of the main question for reduced group C*-algebras of the free groups, and potential attacks on the question of whether or not taking group von Neumann algebras preserves elementary equivalence.
% Lastly in Section 5: we investigate a relationship between the 2-quantifier theory of matrix ultraproducts and the 2-quantifier theory of the free group factors, consider the elementary equivalence problem for the reduced group C*-algebras of the free groups, and outline an approach to finding groups $\Gamma_1, \Gamma_2$ satisfying $\Gamma_1 \equiv \Gamma_2$ while $L(\Gamma_1) \not \equiv L(\Gamma_2)$.

\subsection{Acknowledgements} 

Many of the initial discussions concerning the results in this paper were with David Sherman; we are truly grateful for letting us include some of his valuable insights.  We would also like to thank Ken Dykema for pointing us to his work on standard embeddings, to Dima Shlyakhtenko for answering some of our questions on free entropy, and to Rizos Sklino for discussing the model theory of the free group with us.  Finally, we thank Srivatsav Kunnawalkam Elayavalli for helpful comments on an earlier version of this paper.

\section{Preliminaries}

\subsection{Some model-theoretic notions}

We present a few necessary definitions, but omit much of the formal model theory; we refer the reader to \cite{FarahHartSherman-MTOA1}, \cite{FarahHartSherman-MTOA2}, \cite{FHS3-EE-and-II1-Factors} for details on the model theory of tracial von Neumann algebras.

\begin{defn}
    Two tracial von Neumann algebras $\cM$ and $\cal{N}$ are \textbf{elementarily equivalent}, written $\cM \equiv \cal{N}$, if $\sigma^{\cM}=\sigma^{\cN}$ for every sentence $\sigma$ in the language of tracial von Neumann algebras. 
\end{defn}

By the continuous version of the Keisler-Shelah isomorphism theorem \cite[Theorem 2.1(2)]{FHS3-EE-and-II1-Factors}, $\cM$ and $\cN$ are elementarily equivalent if and only if there are ultrafilters $\cU$ and $\cal V$ on (possibly uncountable) index sets $I$ and $J$ for which $\cM^\cU\cong \cN^\cal V$.  (If one is willing to assume the continuum hypothesis, then one can assume that $\cU=\cal V$ lives on a countable index set.)

We will also need the notion of an elementary embedding:

\begin{defn}
An embedding $i:\cM\hookrightarrow \cN$ between tracial von Neumann algebras is \textbf{elementary} if $\varphi(\vec a)^{\cM}=\varphi(i(\vec a))^{\cN}$ for all formulae $\varphi(\vec x)$ and all tuples $\vec a$ from $\cM$.
\end{defn}

One can give a similar semantic characterization of elementary embeddings: the embedding  $i:\cM\hookrightarrow \cN$ is elementary if and only if there are ultrafilters $\cU$ and $\cal V$ such that $i$ extends to an isomorphism between $\cM^\cU$ and $\cN^{\cal V}$ (with $\cM$ and $\cN$ viewed as subsets of their respective ultrapowers via the diagonal embedding).

\subsection{Amplifications}

We begin this subsection by recalling the fundamental notions of compression and amplification.

\begin{defn}
    Let $\cM$ be a $II_1$ factor and $0 < t < 1$. The \textbf{compression} of $\mathcal{M}$ by $t$ is the II$_1$ factor $\mathcal{M}_t:= p \mathcal{M} p$, where $p \in \mathcal{P(M)}$ is any projection with $\tau(p) = t$.  (This definition is independent of the choice of such a projection.)
    
    For $1 < t < \infty$, the notion of compression is extended to the notion of \textbf{amplifications of $\mathcal{M}$} by taking tensors with matrix algebras.  More precisely, writing $t = n \cdot \ell$, where $0 < \ell < 1$ and $n \in \bN$, we define
    \[
       \mathcal{M}_t := p \mathcal{M} p \otimes M_n(\bC) \cong M_n(p \mathcal{M} p), 
    \]
     where $p \in \mathcal{P(M)}$ is any projection with $\tau(p) = \ell$.  Once again, this definition is independent of the choice of projection as well as the representation $t=n\cdot \ell$.
\end{defn}

We recall that $(\cal M_s)_t\cong \cal M_{st}$ for any $s,t\in \bb R_+$ (see, for example, \cite[Lemma 4.2.3]{Anantharaman-Popa}).

We record that taking amplifications commutes with taking ultraproducts of $II_1$ factors; this fact can be checked simply by unpacking the definitions (and using the well-known fact that a projection $p$ in an ultraproduct can be represented by a sequence of projections of the same trace as $p$).  In the rest of this paper, $\bR_+$ denotes the set of strictly positive real numbers.

\begin{lem} \label{lemma: compr and ultrapower commute}
   Let $\mathcal{M}$ be a $II_1$ factor. Then for any $t\in \bR_+$ and nonprincipal ultrafilter $\mathcal{U}$, we have
   $
    (\mathcal{M}^\mathcal{U})_t \cong (\mathcal{M}_t)^\mathcal{U}.$

\end{lem}

The second necessary fact is that elementary equivalence is preserved under amplifications. 

\begin{lem} \label{lemma: ee preserved under compression}
   Let $\cM, \cal{N}$ be $II_1$ factors satisfying $\cM \equiv \cal{N}$. Then for any $t\in \bR_+$, we have $\cM_t \equiv \cal{N}_t$.
   
   \begin{proof}
   While one can prove this fact using, for example, Ehrenfeucht-Fra\"isse games, the shortest proof uses the Keisler-Shelah theorem.  Indeed, if $\cM \equiv \cal{N}$, then there are nonprincipal ultrafilters $\cal U$ and $\cal V$ such that
   \begin{equation*} 
        \cM^\mathcal{U} \cong \cal{N}^\mathcal{V}.
   \end{equation*}
     
   By Lemma \ref{lemma: compr and ultrapower commute}, we have
   \[
      (\cM_t)^\mathcal{U} \cong (\cM^\mathcal{U})_t \cong (\cal{N}^\mathcal{V})_t \cong (\cal{N}_t)^\mathcal{V},
   \]
   whence we can conclude $\cM_t \equiv \cal{N}_t$.
   \end{proof}
\end{lem}

Analogous arguments show the following more general fact:

\begin{lem}\label{compresselememb}
If $\alpha: \cM\to \cN$ is an elementary embedding and $p\in \cal P(\cM)$, then so are $\alpha|p \cM p: p \cM p\to \alpha(p)\cN \alpha(p)$ and $\alpha\otimes \operatorname{id}_{M_n(\bb C)}:M_n(\cM)\to M_n(\cN)$.  
\end{lem}

\begin{prop}\label{continuity}
For any II$_1$ factor $\cM$, any $t_k,t\in \mathbb R_+$, and any nonprincipal ultrafilter $\cal U$ on $\bb N$ for which $\lim_{\cU} t_k= t$, we have $\prod_{\cal U} \cM_{t_k} \cong \cM_t^{\cal U}$.  

% This implies that if $t_k \in \cal{F}$ and $t_k \rightarrow t$, then $L(\F_x) \equiv \prod_\cU L(\F_x)_{t_k} \cong L(\F_x)_t^\cU$, so $L(\F_x) \equiv L(\F_x)_t$.
\end{prop}

\begin{proof}
    Note that $(\cM_{t_k})_{\gamma_k} \cong \cM_t$, where $\gamma_k = \frac{t}{t_k}$. Since $\lim_{\cU}\gamma_k= 1$, assume without loss of generality that $\gamma_k < 2$ for all $k\in \bb N$. Set $s > \max\{t, \sup_k t_k\}$. We now define II$_1$ factors $\cal A_k$ as follows:
    \begin{itemize}
        \item if $\gamma_k \leq 1$, set $\cal A_k := p_k \cM_{t_k} p_k$ for some projection $p_k \in \cal P(\cM_s)$ satisfying $\tau(p_k) = \gamma_k$, and
        \item if $1 < \gamma_k < 2$, set $\cal A_k := p_k M_2(\cM_{t_k}) p_k$ for some projection $p_k \in \cal P(\cM_s)$ satisfying $\tau(p_k) = \gamma_k/2$.
    \end{itemize}
    Then, letting $\theta_k: \cal A_k \rightarrow \cM_t$ be the isomorphism witnessing $\cal A_k = (\cM_{t_k})_{\gamma_k} \cong \cM_t$, we have that $\prod_{\cU} \theta_k : \prod_{\cU} \cal A_k \rightarrow \cM_t^\cU$ is an isomorphism. We conclude by noting that $\prod_{\cU} \cal A_k \cong \prod_\cU \cM_{t_k}$.
\end{proof}

% \begin{remark}\label{generalizedcontinuity}
% An analogous proof shows the following more general fact:  For any tracial von Neumann algebra $\cN$, any II$_1$ factor $\cM$, any $t_k,t\in \mathbb R_+$, and any nonprincipal ultrafilter $\cU$ on $\bb N$ with $\lim_\cU t_k= t$, we have $$\prod_{\cal U} (\cN * \cM_{t_k}) \cong (\cN * \cM_t)^{\cal U}.$$  
% \end{remark}

We now recall the definition of the fundamental group of a $II_1$ factor. 
\begin{defn}
    The \textbf{fundamental group} of a $II_1$ factor $\mathcal{M}$ is $$\cal F(\cal M):=\{ t \in \bR_+ \ : \ \mathcal{M}_t \cong \mathcal{M} \}.$$ 
\end{defn}

The terminology fundamental group is appropriate as $\cal F(\cal M)$ is indeed a subgroup of the multiplicative group $\bb R_+$ using the aforementioned fact that $(\cal M_s)_t\cong \cal M_{st}$.

In \cite{GoldbringHart16}, the first-named author and Hart introduced the \textbf{first-order fundamental group} $\mathcal{F}_\fo(\mathcal{M}):=\{t \in \bR_+ \ : \ \mathcal{M}_t \equiv \cM\}$ of a II$_1$ factor $\cal M$. (Equivalently, assuming the continuum hypothesis and further assuming that $\mathcal{M}$ is separable, the first-order fundamental group of a II$_1$ factor $\cM$ is equal to the ordinary fundamental group of any nonprincipal ultrapower of $\cM$.) It was shown there that $\cal F_{fo}(\cal M)$ is a closed subgroup of $\bb R_+$ containing the ordinary fundamental group $\cal F(\cal M)$ as a subgroup.  (We note that $\cal F(\cal M)$ need not always be closed.)  At the time of writing of this paper, there is no known example of a II$_1$ factor $\cal M$ whose first-order fundamental group is not all of $\bb R_+$.

\subsection{Interpolated free group factors}

We next recall the interpolated free group factors:
\begin{fact}[\cite{Dykema94}, \cite{Rad1992fundamental}] \label{interpfact}
    There is a family $\{L(\F_r)\}_{1<r \leq \infty}$ of separable $II_1$ factors with the following properties:
    \begin{enumerate}
        \item If $r\in \{2,3,\ldots\}\cup\{\infty\}$, then $L(\bb F_r)$ is the usual group von Neumann algebra associated to the free group $\bb F_r$.
        \item $L(\bb F_r)*L(\bb F_s)\cong L(\bb F_{r+s})$ for all $r,s\in (1,\infty]$.
        \item $L(\bb F_r)_t\cong L(\bb F_{1+\frac{r-1}{t^2}})$ for all $r\in (1,\infty]$ and $t\in (0,\infty)$.
    \end{enumerate}
\end{fact}

A particular consequence of item (3) above is that the interpolated free group factors $L(\bb F_r)$ for finite $r$ are each amplifications of each other, whence they all have the same fundamental group, while the fundamental group of $L(\bb F_\infty)$ is $\bb R_+$, a fact first proven by R\v adulescu in \cite{Rad1992fundamental}.

We will need a generalization of item (2) in the previous fact due to Dykema \cite{Dykema93FreeDim}.  To state the result, we define the \textbf{free dimension} $\operatorname{fd}$ of certain tracial von Neumann algebras as follows:
\begin{itemize}
    \item $\operatorname{fd}(L(\bb F_r))=r$.
    \item $\operatorname{fd}(M_k(\bb C))=1-\frac{1}{k^2}$.
    \item $\operatorname{fd}(L(\bb Z))=\operatorname{fd}(\cal R)=1$.
\end{itemize}

Here and throughout this paper, $\cal R$ denotes the unique separable hyperfinite II$_1$ factor.

\begin{fact}\label{freedimfact}
If $\cal M$ and $\cal N$ are each either $L(\bb F_r)$ (for some $r\in (1,\infty]$), $M_k(\bb C)$ (for some $k\in \bb N$), $L(\bb Z)$, or $\cal R$, then $\cal M*\cal N\cong L(\bb F_t)$ with $t:=\operatorname{fd}(\cal M)+\operatorname{fd}(\cal N)$.
\end{fact}

Thus, for example, we have $\cal R*L(\bb F_n)\cong L(\bb F_{n+1})$ and $L(\bb F_k)*M_n(\bb C)\cong L(\bb F_{k+1-\frac{1}{n^2}})$.

We will also have the occasion to use an even more general formula due to Dykema \cite[Theorem 1.2] {Dykema93FreeDim}:

\begin{fact}\label{dykemafact}
For any two II$_1$ factors $\cal M$ and $\cal N$ and any $n\geq 1$, we have
$$((\cal M\otimes M_n(\bb C))*\cal N)_{\frac{1}{n}}\cong \cal M*(M_n(\bb C)*\cal N)_{\frac{1}{n}}.$$
\end{fact}

Using Facts \ref{freedimfact} and \ref{dykemafact}, we obtain:

\begin{fact}\label{freeprodfact}
For any two II$_1$ factors $\cal M$ and $\cal N$ and any $n\geq 1$, we have 
$$(\cal M\otimes M_n(\bb C))*(\cal N\otimes M_n(\bb C))\cong (\cal M*\cal N*L(\bb F_{n^2-1}))\otimes M_n(\bb C).$$
\end{fact}

Proposition \ref{continuity} and Fact \ref{interpfact} together imply the following:

\begin{cor} \label{cor: continuity}
For any $r_k,r\in (1,\infty)$ with $r_k\to r$, we have $\prod_{\cal U}L(\bb F_{r_k})\cong L(\bb F_r)^{\cal U}$.
\end{cor}

Given the previous corollary, the next question becomes natural:

\begin{question}\label{freeultraquestion}
If $\lim_\cal U r_k=\infty$, do we have $\prod_{\cal U} L(\bb F_{r_k})\cong L(\bb F_\infty)^\cal U$?
\end{question}

The following question at the other extreme was posed to us by David Sherman:

\begin{question}\label{anotherfreeultraquestion}
If $\lim_\cal U r_k=1$, what can one say about $\prod_{\cal U} L(\bb F_{r_k})$?
\end{question}

\section{Towards a dichotomy for free group factor elementary equivalence}

In direct analog to Theorem \ref{thm: fgf-alternative}, we have the following conjectured dichotomy for free group factor elementary equivalence:

\begin{conj} \label{thm: fgf-ee-alternative}
One of the following two statements must hold.
\begin{enumerate}
       \item $L(\F_r) \equiv L(\F_s)$ for all $1 < r \leq s \leq \infty$, and the first-order fundamental group of $L(\F_r)$ is $\bR_+$ for all $1 < r \leq \infty$.
       \item $L(\F_r) \not \equiv L(\F_s)$ for all $1 < r < s \leq \infty$, and the first-order fundamental group of $L(\F_r)$ is $\{1\}$ for all $1 < r < \infty$.
   \end{enumerate}
\end{conj}

In this section, we show how a general trichotomy always holds and explain under what additional hypotheses we can establish Conjecture \ref{thm: fgf-ee-alternative} above. 

\subsection{A trichotomy}

While we are currently unable to establish the desired dichotomy appearing in Conjecture \ref{thm: fgf-ee-alternative}, we can prove an interesting trichotomy for the first-order fundamental group of the interpolated free group factors.

% \begin{lem} \label{lemma: fo fund closed}
%     If $x_n \rightarrow x$ in $\bR_+$, then $L(\F_x) \equiv \prod_\cU L(\F_{x_n})$.
% \end{lem}

% \begin{proof}
% For $t_n = \frac{x_n - 1}{x - 1}$, we have $L(\F_{x_n})_{t_n} = L(\F_x)$. As $t_n \rightarrow 1$, we can assume $t_n < 2$ for all $n$. Working inside of $L(\F_{x+1})$, we can find projection $p_n$ of the right trace so that
% $A_n := p_n L(\F_{x_n}) p_n = L(\F_x)$ for $t_k \leq 1$, and
% $A_n := p_n M_2(L(\F_{x_n})) p_n = L(\F_x)$ for $t_n > 1$.
% Let $\theta_n : A_n \rightarrow L(\F_x)$ be an isomorphism for each $n$. Then $\prod_\mathcal{U} \theta_n : \prod_\mathcal{U} A_n \rightarrow L(\F_x)^\mathcal{U}$ is an isomorphism, and $\prod_\cU A_n \cong \prod_\cU L(\F_{x_n})$.
% \end{proof}

\begin{thm} \label{thm: trichotomy}
Exactly one of the following holds:
\begin{enumerate}
    \item $L(\F_r) \equiv L(\F_s)$ for all $1 < r \leq s < \infty$, and $\mathcal{F}_{fo}(L(\F_r)) = \bR_+$ for all $1 < r \leq \infty$.
    \item $L(\F_r) \not \equiv L(\F_s)$ for all $1 < r < s < \infty$, and $\mathcal{F}_{fo}(L(\F_r)) = \{1\}$ for all $1 < r < \infty$.
    \item There is $\alpha\in (1,\infty)$ such that $\mathcal{F}_{fo}(L(\F_r))=\alpha^\bb Z$ for all $1<r < \infty$. 
\end{enumerate}
\end{thm}

\begin{proof}
    Suppose that $L(\F_r)\equiv L(\F_s)$ for some $1 < r < s < \infty$, and set $t:= (\frac{r-1}{s-1})^{1/2}$. Let us call such a $t$ a \textbf{ratio}.
Then for any $x,y\in (1,\infty)$ for which $t = (\frac{x-1}{y-1})^{1/2}$, we have $L(\F_x)\equiv L(\F_y)$.

Further suppose that we have two ratios $t_1$ and $t_2$ that are not simply powers of each other.  Then the multiplicative subgroup $G$ they generate is dense in $\bb R_+$. Using Proposition \ref{continuity}, we see that $G$ must also be closed, and hence all of $\bR_+$. 
% Given any $x \in \bb R_+$, there is a sequence $x_n$ from $G$ such that $x_n\to x$. By Lemma \ref{lemma: fo fund closed}, $L(\F_x) \equiv \prod_{\cal U} L(\F_{x_n})$.  But $L(\F_y)\equiv L(\F_z)$ for any $y,z \in G$ by the previous paragraph, whence $L(\F_x)$ also has the same theory as the $L(\F_y)$'s for $y\in G$.
\end{proof}

A consequence of the previous theorem is that all interpolated free group factors have the same first-order fundamental group, which we will denote by $\mathcal{F}$ in the remainder of this paper.

\subsection{Including $L(\bb F_\infty)$}

Unlike its classical counterpart, $L(\F_\infty)$ is notably missing in Theorem \ref{thm: trichotomy} above.  We discuss a couple of hypotheses that would remedy this fact.

\begin{lem}\label{infinity}
If $L(\bb F_r)\equiv L(\bb F_\infty)$ for some $r\in (1,\infty)$, then $L(\bb F_r)\equiv L(\bb F_s)\equiv L(\bb F_\infty)$ for all $s\in (1,\infty)$.
\end{lem}

\begin{proof}
Fix $s\in (1,\infty)$ and take $t\in \bb R_+$ such that $L(\bb F_r)_t\cong L(\bb F_s)$.  Using Lemma \ref{lemma: ee preserved under compression} and the fact that $L(\bb F_\infty)$ has full fundamental group, we have that
$$L(\bb F_s)\cong L(\bb F_r)_t\equiv L(\bb F_\infty)_t\cong L(\bb F_\infty).\qedhere$$
\end{proof}

Notice that if $L(\bb F_r)\equiv L(\bb F_\infty)$ for all $r\in (1,\infty)$, then Question \ref{freeultraquestion} has a positive answer.  On the other hand: 

\begin{lem}
If Question \ref{freeultraquestion} has a positive answer and $\cal F \not=\{1\}$, then $L(\bb F_r)\equiv L(\bb F_\infty)$ for all $r\in (1,\infty)$.
\end{lem}

\begin{proof}
Fix $t\in \cal F \cap (0,1)$ and notice that $$L(\bb F_2) \equiv L(\bb F_2)^\cU \equiv \prod_\cal U L(\bb F_2)_{t^k}\cong  \prod_{\cal U} L(\bb F_{1+\frac{1}{t^{2k}}}) \cong L(\bb F_\infty)^\cal U,$$
where the first and second equivalences follows from \L os' theorem and the last isomorphism follows from the assumption that Question \ref{freeultraquestion} has a positive answer.
\end{proof}

We now examine one other avenue from which we can obtain that all interpolated free group factors, including $L(\F_\infty)$, are elementarily equivalent. It requires the notion of a standard embedding of free group factors, taken from \cite[Definition 4.1]{Dykema93FreeDim}.

\begin{defn}
    Let $r \leq r'$. Then $\psi: L(\F_r) \rightarrow L(\F_{r'})$ is a \textbf{standard embedding} if:  for some tracial $W^*$-probability space $(\cM, \tau)$, with $\cM$ containing a copy of $\mathcal{R}$ and a semicircular family $\omega = \{X^t \mid t \in T\}$ such that $\mathcal{R}$ and $ \omega$ are free, there exist subsets $S \subseteq S' \subseteq T$, projections $p_s \in \cal P(\mathcal{R})$ for $s \in S'$, and isomorphisms $\alpha: L(\F_r) \rightarrow (\mathcal{R} \cup \{p_s X^s p_s \mid s \in S\})''$ and $\beta: L(\F_{r'}) \rightarrow (\mathcal{R} \cup \{p_s X^s p_s \mid s \in S' \})''$ such that $\psi = \beta^{-1} \circ i \circ \alpha$, where $i$ is the inclusion map. In other words, the following diagram commutes: \\
    \begin{center}
    \begin{tikzcd}
            (\mathcal{R} \cup \{p_s X^s p_s\}_{s \in S})'' \arrow[rr, "i"]                         &  &  (\mathcal{R} \cup \{p_s X^s p_s\}_{s \in S'})''             \\
            L(\F_r) \arrow[u, "\alpha"] \arrow[rr, "\psi"'] &  & L(\F_{r'}) \arrow[u, "\beta"]
    \end{tikzcd}
    \end{center}
\end{defn}

\begin{lem} \label{lemma: induced standard emb}
If $\alpha:L(\bb F_r)\to L(\bb F_s)$ is standard, $t\in (0,1)$ is such that $L(\bb F_r)_t\cong L(\bb F_s)$, and $p\in \cal P(L(\bb F_r))$ is a projection with $\tau(p)=t$, then $\alpha|pL(\bb F_r)p$ is a standard embedding that induces a standard embedding $\alpha':L(\bb F_s)\to \alpha(p)L(\bb F_s)\alpha(p)$.  If $\alpha$ is elementary, then so is $\alpha'$.
\end{lem}

\begin{proof}
The fact that $\alpha|pL(\bb F_r)p$ is standard is  \cite[Proposition 4.2]{Dykema93FreeDim}.  Now compose with an isomorphism $L(\bb F_s)\to pL(\bb F_r)p$ to obtain $\alpha'$.  That $\alpha'$ is elementary if $\alpha$ is elementary is a consequence of Lemma \ref{compresselememb}.
\end{proof}

\begin{thm}
Suppose that there is an embedding $\alpha:L(\bb F_r)\to L(\bb F_s)$ that is both standard and elementary.  Then $L(\bb F_x)\equiv L(\bb F_\infty)$ for all $x\in (1,\infty)$.
\end{thm}

\begin{proof}
Starting with the given embedding $\alpha$, one can iterate the previous lemma to get an inductive sequence of standard elementary embeddings $L(\bb F_{r_i})\to L(\bb F_{r_{i+1}})$ with $\lim_i r_i=\infty$.  By \cite[Proposition 4.3(ii)]{Dykema93FreeDim}, the limit is isomorphic to $L(\bb F_\infty)$.  Moreover, since the chain is elementary, we have that $L(\bb F_r)\equiv L(\bb F_\infty)$.  The theorem now follows using Lemma \ref{infinity}.
\end{proof}

In particular, since the natural inclusions $L(\bb F_r)\to L(\bb F_r)*L(\bb F_s)$ are standard \cite[Proposition 4.4(i)]{Dykema93FreeDim} (e.g. the natural inclusions $L(\bb F_m)\to L(\bb F_n)$ for integers $m\leq n$), if any of them are elementary, we get that $L(\bb F_x)\equiv L(\bb F_\infty)$ for all $x\in (1,\infty)$.

\subsection{Free Products That Preserve Elementary Equivalence} \label{sec: dichotomy conj}
We record here that the free group factor alternative for elementary equivalence can be obtained in much the same way as in \cite[Theorem 5.1]{MunsterLectures} if taking free products with a fixed tracial von Neumann algebra preserves first-order theories.

\begin{question}
    Does taking the free product with a fixed tracial von Neumann algebra preserve first-order theories?
    That is, if we have two tracial von Neumann algebras $\cM \equiv \cN$, and a third tracial von Neumann algebra $\mathcal{A}$, is it true that $\cM * \mathcal{A} \equiv \cN * \mathcal{A}$?
\end{question}

We obtain the free group factor alternative for elementary equivalence 
in the case that the previous question has a positive resolution (in particular for $\cal{A} = L(\F_{2 - \epsilon})$).

\begin{thm} 
If taking the free product with $L(\F_{2 - \epsilon})$ preserves elementary equivalence for all $\epsilon\in (0,1)$, then Conjecture \ref{thm: fgf-ee-alternative} has a positive solution.

\begin{proof}
%   We follow along with the same argument as the proof given in \cite[Theorem 5.1]{MunsterLectures}, adapting to elementary equivalence as needed.

   \noindent 
   Since $L(\F_r) * L(\F_{r'}) \cong L(\F_{r + r'})$ for all $1 < r, r' \leq \infty$, the same holds true of elementary equivalence:
   \begin{equation} \label{eqn: addition for ee}
       L(\F_r)* L(\F_{r'}) \equiv L(\F_{r + r'}), \text{ for } 1 < r, r' \leq \infty.
   \end{equation}
   Similarly, by Fact \ref{interpfact} above, we have
   \begin{equation*}
        L(\F_r)_t \equiv L \left(\F_{1 + \frac{r-1}{t^2}} \right) \text{ for all } 1 < r \leq \infty, \ 0 < t < \infty.
   \end{equation*}

Now suppose that there is $1 < r < s < \infty$ satisfying $L(\F_r) \equiv L(\F_s)$. Then for all $0 < t < \infty$, 
    \begin{equation} \label{eqn: exchange r and s}
        L(\F_{1 + t^{-2}(r-1)}) \equiv L(\F_r)_t \equiv L(\F_s)_t \equiv L(\F_{1 + t^{-2}(s-1)}),
    \end{equation}
    where the center equivalence is due to Lemma \ref{lemma: ee preserved under compression}.
    
    \noindent 
    Choose $0 < \epsilon < 1$ and $0 < t < \infty$ satisfying $\epsilon = t^{-2} (r-1)$. Then,
    \[
        1 + t^{-2}(s-1) = 1 + \epsilon \left( \frac{s-1}{r-1} \right).
    \]
    Thus, 
    \begin{align*}
        L(\F_3) &\equiv L(\F_{1+\epsilon}) * L(\F_{2 - \epsilon})  &\text{ by (\ref{eqn: addition for ee})}\\
        &= L(\F_{1 + t^{-2}(r-1)}) * L(\F_{2 - \epsilon}) \\
        &\equiv L(F_{1 + t^{-2}(s-1)}) * L(\F_{2 - \epsilon}) &\text{ by (\ref{eqn: exchange r and s}) and Q3.3} \\
        &= L(\F_{1 + \epsilon (\frac{s-1}{r-1})}) * L(\F_{2 - \epsilon}) \\
        &\equiv L(\F_{3 + \epsilon (\frac{s-r}{r-1})})  &\text{ by (\ref{eqn: addition for ee})} \\
        &\equiv L(\F_3)_\alpha,
    \end{align*}
    where $\alpha^{-2} = 1 + \frac{\epsilon}{2} (\frac{s-r}{r-1})$. 
    Then, letting $\epsilon$ vary over $(0,1)$, we have that $(1, 1 + \frac{1}{2}(\frac{s-r}{r-1})) \subseteq \mathcal{F}_\fo(L(\F_3))$. Thus, $\mathcal{F}_\fo(L(\F_3))$ is a multiplicative subgroup of $\bR_+$ containing an open interval, and is therefore all of $\bR_+$. 
    We conclude that $L(\F_3) \equiv L(\F_x)$ for all $x \in (1, \infty)$.
\end{proof}
\end{thm}

% By Remark \ref{generalizedcontinuity}, if free producting with $L(\bb F_{2-\epsilon})$ preserves elementary equivalence for all $\epsilon\in (0,1)$, then free producting with $L(\bb F_2)$ also preserves elementary equivalence.

We next explore some results which show that if taking the free product with $L(\bb Z)$ or $\cal R$ preserves elementary equivalence, then under certain conditions, the trichotomy can be improved to a dichotomy.

\begin{defn}
Call $(r,s)\in (1,\infty)^2$ an \textbf{independent pair} if the numbers $\{\frac{r+k}{s+k}\}_{k=-1}^{\infty}$ do not all generate the same multiplicative subgroup of $\bb R_+$.
\end{defn}

\begin{lem}
Suppose that there is an independent pair $(r,s)\in (1,\infty)^2$ for which $L(\bb F_r)\equiv L(\bb F_s)$.  Further suppose that taking the free product with either $L(\bb Z)$ or $\cal R$ preserves elementary equivalence.  Then $L(\bb F_x)\equiv L(\bb F_y)$ for all $x,y\in (1,\infty)$.
\end{lem}

\begin{proof}
If $L(\bb F_r)\equiv L(\bb F_s)$, then by Fact \ref{freedimfact} and the assumptions of the lemma, we have that $L(\bb F_{r+i})\equiv L(\bb F_{s+i})$ for all $i\geq 1$.
%  $L(\bb F_r)*L(\bb Z)\cong L(\bb F_{r+1})$ is true for all $r\geq 1$; see Lemma 1.6 in the Dykema paper I referred to.  It also works with free producting with $\cal R$.  This is summarized in Theorem 5.2 of Dykema's Munster survey.
Then by the assumption on $r$ and $s$, we get that not all of the associated ratios generate the same subgroup and then we conclude as in the remarks preceding Theorem \ref{thm: trichotomy}.
\end{proof}

\begin{cor}
Suppose that $L(\F_2) \equiv L(\F_n)$ for some integer $n \neq 2$ and that taking the free product with either $L(\bb Z)$ or $\cal R$ preserves elementary equivalence.  Then $L(\bb F_r)\equiv L(\bb F_s)$ for all $r, s \in (1,\infty)$.
\end{cor}

\begin{proof}
First suppose that $n$ is odd.  Then $(2,n)$ is an independent pair since $\frac{1}{n-1}$ cannot be a power of $\frac{2}{n}$.

On the other hand, suppose $L(\bb F_2) \equiv L(\bb F_n)$ for some even integer $n = 2\ell$, so that $\frac{2}{n} = \frac{1}{\ell}$. Assuming that the multiplicative subgroup generated by $\frac{1}{n-1}, \frac{1}{\ell}, \frac{3}{n+1}, \ldots$ are all the same, then there exists some $m$ so that 
\[
    \frac{1}{\ell^m} = \frac{3}{n+1} \implies n+1 = 3 \ell^m \implies  n \equiv 2 \mod 3.
\]
The same argument with $\frac{1}{\ell^m} = \frac{k}{n+(k-2)}$ shows that $n \equiv 2 \mod k$ for every $k \geq 3$. This implies $n = 2$.
\end{proof}

We examine a similar line of reasoning using free products by matrix algebras.  We first note that Fact \ref{freedimfact} implies the following:

\begin{lem}
Suppose that $L(\bb F_r)\equiv L(\bb F_s)$ and taking the free product with $M_k(\bb C)$ preserves theories.  Then $L(\bb F_{r+1-\frac{1}{k^2}})\equiv L(\bb F_{s+1-\frac{1}{k^2}})$.
\end{lem}

\begin{prop}
    If taking the free product with $M_2(\bb C)$ preserves elementary equivalence and $L(\F_m) \equiv L(\F_n)$ where $m, n$ are integers with different parity, then $L(\bb F_r)\equiv L(\bb F_s)$ for all $r, s \in (1,\infty)$. 
\end{prop}

\begin{proof}
By taking the free product with $M_2(\bb C)$ once we have $L(\F_{m + \frac{3}{4}}) \equiv L(\F_{n + \frac{3}{4}})$, which yields the ratio $\frac{4m-1}{4n-1}$. If this ratio generates the same multiplicative subgroup as $\frac{m-1}{n-1}$, then we obtain that there is some $t$ such that $(4m-1)(n-1)^t = (4n-1)(m-1)^t$. Since $4m-1$ and $4n-1$ are odd, and $m,n$ have different parity, this yields a contradiction. 
\end{proof}

% \textcolor{blue}{This is great!  Although it assumes $t>0$.  If $t<0$ then just move the terms to the other sides, but it still works.}

Although it seems difficult to prove that taking the free product can ever preserve elementary equivalence, in some cases we have some positive results, as in the following proposition. 
% The following proposition follows immediately from Lemma \ref{lemma: compr and ultrapower commute} and Fact \ref{dykemafact}.

% \begin{lem}
% Suppose that $A,A',B,B'$ are tracial von Neumann algebras such that $A\cong A\otimes M_2(\bb C)$ (and similarly for $A'$, $B$, and $B'$).  Further assume that $A*B\equiv A'*B'$.  Then $A*B*L(\bb F_3)\equiv A'*B'*L(\bb F_3)$.
% \end{lem}

% \begin{proof}
% $A*B\equiv A'*B'$ implies $(A*B)_{\frac{1}{2}}\equiv (A'*B')_{\frac{1}{2}}$.  But Proposition 4.11 in the Dykema survey yields the desired result.
% \end{proof}

% As an example, if $B$ and $B'$ have $2$ in their fundamental group and $\cal R*B\equiv \cal R*B'$, then $L(\bb F_4)*B\equiv L(\bb F_4)*B'$.  In particular, if $\cal R\equiv \cal R'$ and $L(\bb F_2)\equiv \cal R*\cal R'$, then $L(\bb F_5)\equiv L(\bb F_4)*\cal R'$.

% I think this generalizes:

\begin{prop}
Suppose that $\cal M,\cal M',\cal N,\cal N' $ are II$_1$ factors such that $\cal M\cong \cal M\otimes M_n(\bb C)$ and similarly for $\cal M'$, $\cal N$, and $\cal N'$.  Further assume that $\cal M*\cal N\equiv \cal M'*\cal N'$.  Then $\cal M*\cal N*L(\bb F_{n^2-1})\equiv \cal M'*\cal N'*L(\bb F_{n^2-1})$.
\end{prop}

\begin{proof}
We first claim that
$$((\cM \otimes M_n(\bb C))*(\cN \otimes M_n(\bb C))_{1/n}\cong \cM * \cN *L(\bb F_{n^2-1}).$$

Indeed, by applying Fact \ref{dykemafact} twice, we have
\begin{align*}
    ((\cM \otimes M_n(\bb C)) * (\cN \otimes M_n(\bb C)))_{1/n} &\cong \cM * ((\cN \otimes M_n(\bb C)) * M_n(\bb C))_{1/n} \\
    &\cong \cM * (\cN * (M_n(\bb C) * M_n(\bb C))_{1/n}).
\end{align*}
Now Fact \ref{freedimfact} implies that $M_n(\bb C) * M_n(\bb C) \cong L(\F_{2 - \frac{2}{n^2}})$. Then $L(\F_{2- \frac{2}{n^2}})_{1/n} \cong L(\F_{n^2 - 1})$ gives the desired conclusion.

The result of the proposition follows by using the assumptions to obtain
\begin{align*}
     \cM * \cN *L(\bb F_{n^2-1}) \cong ((\cM \otimes M_n(\bb C))*(\cN \otimes M_n(\bb C))_{1/n} \cong (\cM*\cN)_{1/n} 
     \equiv (\cM'*\cN')_{1/n}, 
\end{align*}
and then performing the same equivalences to obtain $(\cM'*\cN')_{1/n} \cong \cM' * \cN' *L(\bb F_{n^2-1})$.
\end{proof}

A particular instance of the previous proposition is the case that $\cal M$, $\cal M'$, $\cal N$ and $\cal N'$ are II$_1$ factors, each of which have all positive integers in their fundamental group (e.g. if they have full fundamental group).  In this case, if $\cal M*\cal N\equiv \cal M'*\cal N'$, then $\cal M*\cal N*L(\bb F_{n^2-1})\equiv \cal M'*\cal N'*L(\bb F_{n^2-1})$ for all $n\geq 1$.

% \begin{cor}
% Suppose that $L(\bb F_r)\equiv L(\bb F_s)$.  Further suppose that there are $A$, $B$, $C$, and $D$ and $k\in \bb N$ such that $A$, $B$, $C$, and $D$ are all either matrix algebras, interpolated free group factors, $L(\bb Z)$, or $\cal R$.  Suppose also that $(A\otimes M_n(\bb C))*(B\otimes M_n(\bb C))\cong L(\bb F_r)$ and $(C\otimes M_n(\bb C))*(D\otimes M_n(\bb C))\cong L(\bb F_s)$.  Then $A*B*L(\bb F_{n^2-1})\equiv C*D*L(\bb F_{n^2-1})$.
% \end{cor}

% If the ratios $\frac{r-1}{s-1}$ and $\frac{r'-1}{s'-1}$, where $r' = \operatorname{fdim}(A*B*L(\F_{k^2-1})$ and $s' = \operatorname{fdim}(C*D*L(\F_{k^2-1})$ do not lie in the same multiplicative subgroup, then we have that they are all elementarily equivalent. Unfortunately, some computation shows that this does not happen; PERHAPS DELETE this Proposition unless we find some other use for it.

\section{$\forall \exists$-theories of free group factors}

In this section, we establish some promising results concerning the values of $\forall\exists$-sentences in free group factors.  Recall that a $\forall\exists$-sentence is a sentence $\sigma$ of the form $\sup_{\vec x}\inf_{\vec y} \varphi(\vec x,\vec y)$, where $\varphi$ is a quantifier-free formula.  For a II$_1$ factor $\cal M$, we consider the \textbf{$\forall\exists$-theory} $\Th_{\forall\exists}(\cal M)$ of $\cal M$, which is the function $\sigma\mapsto \sigma^{\cal M}$ defined on the set of all $\forall\exists$ sentences.

Our results concerning $\forall\exists$-sentences will follow from the existence of certain nice embeddings between some pairs of interpolated free group factors.

% \begin{cor}
% Suppose that $r,s\in (1,\infty)$ are such that $r+1\leq s$.  Then for any $\forall\exists$-sentence $\sigma$, we have $\sigma^{L(\bb F_r)}\leq \sigma^{L(\bb F_s)}$.
% \end{cor}

% \begin{proof}
% If $r+1<s$, then the result follows from the previous proposition using that $L(\bb F_s)=L(\bb F_r)*L(\bb F_{s-r})$.  If $s=r+1$, then the result follows from the previous case by considering any sequence $s_n\to s$ with $s_n>r+1$ and noting $\sigma^{L(\bb F_{s_n})}\to \sigma^{L(\bb F_s)}$.
% \end{proof}

% Here are some cases not covered by the previous corollary:

% \begin{prop}
% For any $n\geq 2$ (probably any $r\in [2,\infty)$), there is an existential embedding $L(\bb F_n)\hookrightarrow L(\bb F_{n+1-\frac{1}{k^2}})$.
% \end{prop}

% \begin{proof}
% This follows from Theorem 4.4 in Dykema, which states that $$L(\bb F_n)*M_k(\bb C)\cong L(\bb F_{nk^2})\otimes M_k(\bb C)\cong L(\bb F_{n+1-\frac{1}{k^2}}).$$
% \end{proof}

\subsection{Existential embeddings}

\begin{defn}
    An embedding $i:\cM \hookrightarrow \cN$ is \textbf{existential} if $\varphi(\vec a)^{\cal M}=\varphi(i(\vec a))^{\cal N}$ for all existential formulae $\varphi(\vec x)$, that is, all formulae of the form $\varphi(\vec x)=\inf_{\vec y}\psi(\vec x,\vec y)$, where $\psi$ is quantifier-free. 
\end{defn}

 Equivalently, $i:\cM \hookrightarrow \cN$ is existential if there is an embedding $j:\cN\hookrightarrow \cM^\cU$ for which $j\circ i$ restricts to the diagonal embedding of $\cM$ into $\cM^\cU$.

Since existential embeddings preserve formulae with one quantifier, we write $i:\cM \hookrightarrow_1 \cN$ to emphasize that the embedding $i$ is existential. We may also write $\cal M\hookrightarrow_1 \cal N$ to mean that there is an existential embedding of $\cal M$ into $\cal N$.  It is clear that whenever $\cM \hookrightarrow_1 \cN$, then $\Th_{\forall \exists}(\cM) \leq \Th_{\forall \exists}(\cN)$ (as functions). 

\begin{lem}\label{ecfree}
If $i:\cM \hookrightarrow_1 \cal P$ and $j:\cN \hookrightarrow_1 \cal Q$ are existential embeddings, then so is $i*j:\cM* \cal N \hookrightarrow_1 \cal P* \cal Q$.
\end{lem}

\begin{proof}
Let $i':\cal P\hookrightarrow \cal M^\cal U$ and $j':\cal Q\to \cal N^\cal U$ be such that $i'\circ i$ and $j'\circ j$ are the respective diagonal embeddings.  The lemma follows from observing that the composition of the natural maps
$$\cM * \cN \hookrightarrow \cal P* \cal Q\hookrightarrow \cM^\cal U * \cN^\cal U\subseteq (\cM* \cN)^\cal U$$ is the diagonal embedding, where the first map is $i*j$ and the second map is $i'*j'$. 
%CHECK FOR LAST INCLUSION:
% We check that $M^\cU * N^\cU \subseteq (M * N)^\cU$.
% We have that both $M^\cU$ and $N^\cU$ are von Neumann subalgebras of $(M*N)^\cU$, so we just have to check that they are free. 
% For any $n \in \bN$, let $\{ \bar{x}^{(j)} := (x^{(j)}_i)_\cU \}_{j=1}^n$ be a set of alternating elements from $M^\cU$ and $N^\cU$ (i.e. if $\bar{x}^{(1)} \in M^\cU$, then $\bar{x}^{(j)} \in N^\cU$ for all even $j$, and $\bar{x}^{(j)} \in M^\cU$ for all odd $j$).
% Also suppose that $\tau^{(M*N)^\cU}(\bar{x}^{(j)}) = 0$ for all $j$.
% Then, 
% \[
%     \tau^{(M*N)^\cU}(\bar{x}^{(1)} \cdots \bar{x}^{(n)}) = \lim_{i \rightarrow \cU} \tau^{M *N}(x^{(1)}_i \cdots x^{(n)}_i) = 0.
% \]
% Justifying the last equality: since $\tau^{(M*N)^\cU}(\bar{x}^{(j)}) = 0$ for all $j$, then for each $j$ there is some set $A_j \in \cU$ such that $\tau^{M*N}(x_i^{(j)}) = 0$ for all $i \in A_j$. 
% Thus, for all $i \in \bigcap_{j=1}^n A_j$, we have $\tau^{M*N}(x_i^{(j)}) = 0$ for all $j$. As $\cU$ is closed under finite intersection, we can conclude that $\tau^{M*N}(x^{(1)}_i \cdots x^{(n)}_i) = 0$ for $\cU$-many indices $i$.
\end{proof}

In what follows, we call a tracial von Neumann algebra \textbf{embeddable} if it admits a trace-preserving embedding into $\cal R^\cU$.

\begin{prop}\label{freegroupexistentia}
For any $n\geq 2$ and any embeddable type II$_1$ von Neumann algebras $\cM_1,\ldots,\cM_n$, there is an existential embedding $L(\bb F_n) \hookrightarrow_1 \cM_1*\cdots * \cM_n$.
\end{prop}

\begin{proof}
We proceed by induction on $n$.  Suppose $n=2$.  Since each $\cM_i$ is type II$_1$, there is an embedding $\cal R\hookrightarrow \cM_i$ for each $i=1,2$.  By \cite[Lemma 2.1]{FGHS2016ecfactors}, these embeddings are existential.  By Fact \ref{freedimfact} and Lemma \ref{ecfree}, there is an existential embedding $L(\bb F_2) \cong \cal R * \cal R\hookrightarrow_1 \cM_1* \cM_2$.  The induction step follows from Fact \ref{freedimfact}, that is, using that $\cal R*L(\bb F_n)\cong L(\bb F_{n+1})$.
\end{proof}

For the next result, we recall that Connes \cite{Connes76} established that the free group factors are embeddable (this was also shown independently by Wasserman \cite{Wasserman1976}). Since the interpolated free group factors are amplifications of embeddable factors, they are themselves embeddable (as $\cal R$ has full fundamental group).

\begin{cor} \label{cor: fgf-ec-embeddings}
For any natural number $n\geq 2$ and any $r\in \bb R\cup\{\infty\}$ with $r\geq n$, there is an existential embedding $L(\bb F_n) \hookrightarrow_1 L(\bb F_r)$.
\end{cor}

\begin{proof}
If $r=n$, the result is trivial.  Thus, we may suppose that $r>n$.  We may then write $L(\bb F_r)=L(\bb F_{r_1})*\cdots*L(\bb F_{r_n})$ with $r_1,\ldots,r_n\in (1,\infty]$. Since each $L(\bb F_{r_i})$ is embeddable, the result follows from Proposition \ref{freegroupexistentia}.
\end{proof}

\begin{cor}\label{supcor}
$\Th_{\forall\exists}(L(\bb F_\infty)) =\sup_n \Th_{\forall\exists}(L(\bb F_n))$.
\end{cor}

\begin{proof}
The inequality $\Th_{\forall\exists}L(\bb F_\infty)\geq \sup_n \Th_{\forall\exists}(L(\bb F_n))$ follows from Corollary \ref{cor: fgf-ec-embeddings}.  The other inequality follows from a general model-theoretic fact:  if $\cM$ is the union of a chain $(\cM_n)$ of structures, then $\Th_{\forall\exists}(\cM)\leq \sup_n \Th_{\forall\exists}(\cM_n)$. \footnote{This is a well-known fact, but we include a proof for the sake of the reader.  Fix a quantifier-free formula $\varphi(x,y)$ and set $\sigma:=\sup_x\inf_y\varphi(x,y)$. Fix also $a\in \cal M_n$ for some $n\in \bb N$.  Then $(\inf_y\varphi(a,y))^{\cM}\leq (\inf_y\varphi(a,y))^{\cM_n}\leq \sigma^{\cM_n}$.  Since the set-theoretic union of the $\cM_n$'s is a dense subset of $\cal M$, we have that $\sigma^{\cM}\leq \sup_n\sigma^{\cM_n}$.} 
\end{proof}

We now show how to generalize Corollary \ref{cor: fgf-ec-embeddings} to allow for both factors to be interpolated free group factors.  First, a general lemma:

\begin{lem}\label{amplifyexistential}
Suppose that $i:\cM \hookrightarrow_1 \cN$ is an existential embedding.  Then for any $t>0$, the map $i_t:\cM_t \hookrightarrow_1 \cN_t$ is an existential embedding.
\end{lem}

\begin{proof}
 Fix an embedding $\phi: \cN \hookrightarrow \cM^\cU$ such that $\phi \circ i$ is the diagonal embedding of $\cM$ into $\cM^\cU$. 
    
Choose some natural number $N$ and projection $p \in \cM$ of appropriate trace so that $\cM_t \cong M_N(p \cM p)$. Then define the map 
    \[
        i_t: 
        \cM_t \cong M_N(p \cM p) \hookrightarrow M_N(p \cN p) \cong \cN_t
    \]
    that applies the map $i$ entrywise in the matrices. Similarly define
    \[  
        \phi_t: M_N(p \cN p) \hookrightarrow M_N(p \cM^\cU p) \cong (\cM^\cU)_t\cong (\cM_t)^\cU
    \]
    as the map that applies $\phi$ entrywise in the matrices. Then $\phi_t \circ i_t$ is the diagonal embedding of $\cM_t$ into $(\cM_t)^\cU$, whence $i_t$ is existential.
\end{proof}

\begin{cor}\label{generalexistential}
For any $r,s\in (1,\infty]$ with $r\leq s$, there is an existential embedding of $L(\bb F_r)\hookrightarrow_1 L(\bb F_s)$, whence $\Th_{\forall \exists}(L(\F_r)) \leq \Th_{\forall \exists} (L(\F_s))$.
\end{cor}

\begin{proof}
Take $t>0$ such that $L(\bb F_2)_t\cong L(\bb F_r)$.  Take $s'\in (1,\infty]$ such that $L(\bb F_{s'})_t\cong L(\bb F_s)$; note that $2\leq s'$.  Then amplifying an existential embedding $L(\bb F_2)\hookrightarrow_1 L(\bb F_{s'})$ yields an existential embedding $L(\bb F_r) \hookrightarrow_1 L(\bb F_s)$ by Lemma \ref{amplifyexistential}.
\end{proof}

The above results motivate the following natural question:

\begin{question}\label{AEcoincide}
Is it true that $\Th_{\forall\exists}(L(\bb F_r))=\Th_{\forall\exists}(L(\bb F_s))$ for all $r,s\in (1,\infty]$?
\end{question}

A positive answer to Question \ref{AEcoincide} would follow from a positive answer to the following question:

\begin{question}\label{embeddingquestion}
Are there $1<r<s<\infty$ for which there exists an existential embedding $L(\bb F_s)\hookrightarrow_1 L(\bb F_r)$?  In particular, are there integers $n<m$ for which there is an existential embedding $L(\bb F_m)\hookrightarrow_1 L(\bb F_n)$ or are there $t\in (1,2)$ for which there is an existential embedding $L(\bb F_2)\hookrightarrow L(\bb F_t)$?
\end{question}

\begin{prop}
If Question \ref{embeddingquestion} has a positive answer, then $L(\bb F_y)\hookrightarrow L(\bb F_x)$ for any $1<x<y<\infty$.  In particular, a positive answer to Question \ref{embeddingquestion} implies a positive answer to Question \ref{AEcoincide}.
\end{prop}

\begin{proof}
    For ease of notation, we assume that there is an existential embedding $L(\F_3) \hookrightarrow_1 L(\F_2)$.  Let $1<x < y<\infty$ be arbitrary. Take $t$ such that $L(\F_3)_t = L(\F_y)$. By amplifying the original existential embedding, we have $L(\F_3)_t \hookrightarrow L(\F_2)_t$, and by combining this embedding with the existential embeddings in Corollary \ref{generalexistential}, we have that $\Th_{\forall \exists} (L(F_z))$ is constant for all $z \in [1 + \frac{1}{t^2}, 1 + \frac{2}{t^2}]$. If $x$ lies in this interval, we are done since $1 + \frac{2}{t^2} = y$. Otherwise, we can repeat this argument with $t_1 = \sqrt{2}t$ in place of $t$ to obtain that $\Th_{\forall \exists} (L(F_z))$ is constant for all $z \in [1 + \frac{1}{2t^2}, 1 + \frac{1}{t^2}]$. By iterating this argument with $t_n = \sqrt{2}^nt$, and choosing $n$ so that $x \in [1 + \frac{1}{2^n t^2}, 1 + \frac{1}{2^{n-1} t^2}]$, we may chain together embeddings so that
    \[
        L(\F_y) = L(\F_3)_t \hookrightarrow_1 L(\F_2)_t = L(\F_3)_{t_1} \hookrightarrow_1 \cdots \hookrightarrow_1 L(\F_2)_{t_{n-1}} = L(\F_3)_{t_n} \hookrightarrow_1 L(\F_2)_{t_n}.
    \]
    Finally, $L(\F_2)_{t_n} \hookrightarrow_1 L(\F_x)$ by the choice of $n$ and Corollary \ref{generalexistential}.
\end{proof}

% \begin{proof}
%    Suppose that there are $1<r<s<\infty$ for which there exists an existential embedding $L(\bb F_s)\hookrightarrow_1 L(\bb F_r)$.  Fix $1<x<y<\infty$.  Fix $t>0$ sufficiently small so that $1+\frac{r-1}{t^2}\leq x<y\leq 1+\frac{s-1}{t^2}$.  Using Lemma \ref{amplifyexistential} and Corollary \ref{generalexistential}, we have that $$L(\bb F_y)\hookrightarrow_1 L(\bb F_s)_t\hookrightarrow_1 L(\bb F_r)_t\hookrightarrow_1 L(\bb F_x),$$ as desired.
% \end{proof}

\begin{remark}
    A natural idea for answering the latter two parts of Question \ref{embeddingquestion} would be to attempt to write $L(\bb F_n)$ as a free product of $m$ many II$_1$ factors and to write $L(\bb F_t)$ as a free product of two II$_1$ factors.  However, both of these questions are open at the moment and a positive solution to either would show that free entropy dimension of Voiculescu \cite{Voi-AnaloguesII} is not independent of generators.
\end{remark}

\begin{remark}
Recall from Fact \ref{interpfact} that $L(\bb F_2)\cong (\cal R*\cal R*L(\bb F_3))\otimes M_2(\bb C)$.  We point out that if an existential embedding $L(\bb F_3)\hookrightarrow_1 L(\bb F_2)$ existed, then viewing this embedding as an existential embedding $L(\bb F_3)\hookrightarrow_1 (\cal R*\cal R*L(\bb F_3))\otimes M_2(\bb C)$, we could not have that the image is contained in $\cal R*\cal R*L(\bb F_3)$ since $L(\bb F_3)$ does not have property Gamma.
%and therefore does not have a copy of $M_2(\bb C)$ in its relative commutant in an ultrapower.
\end{remark}

A positive answer to Question \ref{AEcoincide} follows from the group $\cal F$ being nontrivial:

\begin{prop}
Suppose that $\cal F \not=\{1\}$.  Then for all $r,s\in (1,\infty]$, we have $$\Th_{\forall \exists}(L(\F_r)) =\Th_{\forall \exists}(L(\F_s)).$$
\end{prop}

\begin{proof}
Without loss of generality, assume that $r\leq s$.  By Corollary \ref{supcor}, we may also suppose that $s<\infty$.  We already know that $\Th_{\forall \exists}(L(\bb F_r))\leq \Th_{\forall \exists}(L(\bb F_s))$.  Take $t \in \cal F $  sufficiently small such that $r':=1+\frac{r-1}{t^2} >s$.  We then have $\Th_{\forall \exists}(L(\bb F_s))\leq \Th_{\forall \exists}(L(\bb F_{r'}))=\Th_{\forall\exists}(L(\bb F_r)_t)=\Th_{\forall \exists}(L(\bb F_r))$. 
\end{proof}

\subsection{Applying Sela's theorem}

The following is implicit in the proof of \cite[Theorem A]{KE-DiagEmb2022}: 

\begin{lem}\label{existentialinclusiongroup}
Suppose that $i:\Gamma \hookrightarrow_1 \Lambda$ is an existential embedding of groups.  Then $L(i):L(\Gamma) \hookrightarrow_1 L(\Lambda)$ is an existential inclusion of tracial von Neumann algebras.
\end{lem}

\begin{proof}
The inclusions $\Gamma\hookrightarrow \Lambda\hookrightarrow \Gamma^{\cal U}$ yields inclusions of tracial von Neumann algebras $L(\Gamma)\hookrightarrow L(\Lambda)\hookrightarrow L(\Gamma^{\cal U})$.  By \cite[Lemma 3.2]{KE-DiagEmb2022}, the natural map $\Gamma^{\cal U}\to L(\Gamma)^{\cal U}$ yields an embedding $L(\Gamma^{\cal U})\hookrightarrow L(\Gamma)^{\cal U}$. 
The composition of all these embeddings is the diagonal embedding $L(\Gamma)\hookrightarrow L(\Gamma)^{\cal U}$.
% Alternatively, one could argue syntactically.  An approximate witness in $L(\Lambda)$ to an existential statement with parameters in $L(\Gamma)$ might as well be taken to belong to $\bb C[\Lambda]$; but then, this should boil down to group elements in $\Lambda$ which do or do not multiply to the identity, and then since $\Gamma$ is existential in $\Lambda$, these group elements can be found in $\Gamma$.
\end{proof}

The following major result is due to Sela \cite[Theorem 4]{Sela-FGF-EE}:

\begin{fact}\label{selafact}
The canonical embeddings $\bb F_m\hookrightarrow \bb F_n$ for $m\leq n$ are elementary.
\end{fact}

Lemma \ref{existentialinclusiongroup} and Fact \ref{selafact} yield an alternate proof of the fact that $L(\bb F_m)$ embeds existentially into $L(\bb F_n)$ for $m\leq n$ (although it uses a much heavier hammer).  Even better, this argument shows that the canonical embedding is existential.

% \textcolor{blue}{The following two remarks used to be questions.}

\begin{remark}
In \cite{PopaShly-FGFasGroupFactors2020}, Popa and Shylakhtenko construct groups $\Gamma_r$ such that $L(\Gamma_r)=L(\bb F_r)$ for all $r\in (1,\infty)$.  However, these groups have torsion and so cannot be models of $\operatorname{Th}(\bb F_n)$ and thus the previous methods cannot immediately be extended to cover interpolated free group factors.  It remains an open question whether such groups $\Gamma_r$ can be constructed that are also torsion-free.
\end{remark}

\begin{remark}
The group analogue of the first question in Question \ref{embeddingquestion} has a negative solution, as pointed out to us by Rizos Sklinos (which he attributes to Chloe Perin).  Indeed, suppose towards a contradiction that there is an existential embedding $i:\bb F_m\hookrightarrow \bb F_n$ with $m>n$.  Let $a_1,\ldots,a_m\in \bb F_m$ be free generators of $\bb F_m$.  Since $i$ is existential and each $i(a_j)$ can be written as a word in the generators of $\bb F_n$, there are $b_1,\ldots,b_n\in \bb F_m$ such that each $a_j$ can be written as a word in $b_1,\ldots,b_n$.  We then note that $b_1,\ldots,b_n$ generate $\bb F_m$, contradicting the fact that one cannot generate $\bb F_m$ with fewer than $m$ elements. 
\end{remark}

% In connection with the previous remark, we can show that $L(\bb F_5)$ cannot admit an existential embedding in $L(\bb F_2)$.  In fact, something more general holds:

% \begin{lem}
% For any $r\in (1,\infty)$ and any $n\in \bb N$, $n\geq 2$, we have that $L(\bb F_r)$ does not admit an existential embedding in $L(\bb F_{1+\frac{r-1}{n^2}})$. 
% \end{lem}

% \begin{proof}
% This follows from the fact that if $M$ does not have property Gamma, then $M$ does not admit an existential embedding in $M_n(M)$ for any $n\geq 2$.
% \end{proof}

% --- OLD SECTION for e.c. embedding of matrix ultraprods into FGF

\subsection{Applying a theorem of Popa}

One can also prove a special case of Corollary \ref{cor: fgf-ec-embeddings} using the following theorem of Popa, which is a special case of \cite[Corollary 4.4]{Popa-IndepinSubalgOfUP}:

\begin{fact}\label{popafact}
Suppose that $\cN_1$ and $\cN_2$ are separable subalgebras of $\cM^\cU$.  Then there is a unitary $u\in \cM^\cU$ such that $\cN_1$ and $u\cN_2u^*$ are freely independent, whence $\cN_1\vee u\cN_2u^*\cong \cN_1*\cN_2$.
\end{fact}

\begin{cor}\label{popa} 
Suppose that $\cN$ embeds in $\cM^\cal U$.  Then there is an existential embedding $\cM \hookrightarrow_1 \cM* \cN$.  In particular, $\Th_{\forall \exists}(\cM) \leq \Th_{\forall \exists}(\cM * \cN)$.
\end{cor}

\begin{proof}
Without loss of generality, we may assume that $\cN$ is separable and is a subalgebra of $\cM^\cU$.  Apply Fact \ref{popafact} with $\cN_1=\cM$ (viewed as a subalgebra of $\cM^\cU$ via the diagonal embedding) and $\cN_2=\cN$.  Then the inclusion $\cM\subseteq \cM\vee u\cN u^*$ is existential; composing this inclusion with an isomorphism $\cal M\vee u\cal N u^*\cong \cM*\cN$ yields the desired existential inclusion $\cM\hookrightarrow_1 \cM*\cN$.

\end{proof}

Now suppose that $r,s\in (1,\infty)$ with $r+1<s$.  Then since $L(\bb F_{s-r})$ embeds into $\cal R^\cal U$, which in turns embeds into $L(\bb F_r)^\cal U$, we can use Corollary \ref{popa} to infer that there is an existential embedding $L(\bb F_r)\hookrightarrow_1 L(\bb F_r)*L(\bb F_{s-r})\cong L(\bb F_s)$. 

We derive another consequence of Corollary \ref{popa}:

\begin{prop}
For any embeddable II$_1$ factors $\cM$ and $\cN$ for which $2\in \cal F(\cM)\cap \cal F(\cN)$, there is an existential embedding $L(\bb F_{3/2})\hookrightarrow_1 \cM*\cN$, whence $\Th_{\forall \exists}({L(\bb F_{3/2})}) \leq \Th_{\forall \exists}(\cM*\cN)$.
\end{prop}

\begin{proof}
By Fact \ref{freeprodfact} and the assumption that $2 \in \mathcal{F}(\cM) \cap \mathcal{F}(\cN)$, we have $\cM * \cN \cong M_2(\cM) * M_2(\cN) \cong (\cM*\cN*L(\bb F_3))\otimes M_2(\bb C)$.  Since there is an existential embedding $L(\bb F_3)\hookrightarrow_1 \cM*\cN*L(\bb F_3)$, by Lemma \ref{amplifyexistential}, there is an existential embedding $L(\bb F_{3/2})\cong L(\bb F_3)\otimes M_2(\bb C)\hookrightarrow_1 (\cM*\cN*L(\bb F_3))\otimes M_2(\bb C)\cong \cM*\cN$.
\end{proof}

The previous proposition gives an alternate proof that there is an existential embedding $L(\bb F_{3/2})\hookrightarrow_1 L(\bb F_r)$ for any $r\geq 2$.

Corollary \ref{popa} has another interesting consequence:

\begin{cor}\label{maximalAE}
There is a separable embeddable II$_1$ factor $\cM_{\forall\exists}$ such that $\Th_{\forall\exists}(\cN)\leq \Th_{\forall\exists}(\cM_{\forall\exists})$ for all embeddable II$_1$ factors $\cN$.  Consequently, $\Th_{\forall\exists}(\cM_{\forall\exists} \,*\, \cN)= \Th_{\forall\exists}(\cM_{\forall\exists})$ for all embeddable II$_1$ factors $\cN$.
\end{cor}

\begin{proof}
For each $\forall\exists$ sentence $\sigma$, let $r_\sigma:=\sup\{\sigma^{\cM} \ : \ \cM \text{ an embeddable II}_1 \text{ factor}\}$ and let $\cM_\sigma$ be an embeddable II$_1$ factor for which $\sigma^{\cM_\sigma}=r_\sigma$ (which exists by compactness). Note that all embeddable II$_1$ factors embed into each other's ultrapowers. 
Hence we may apply Corollary \ref{popa} to obtain: for any finite collection of $\forall \exists$ sentences $\sigma_1,\ldots,\sigma_n$, we have that $\sigma_i^{\cM}=r_{\sigma_i}$ for all $i=1,\ldots,n$, where $\cM=\cM_{\sigma_1}*\cdots*\cM_{\sigma_n}$.  By compactness, there is an embeddable II$_1$ factor $\cM_{\forall\exists}$ such that $\sigma^{\cM_{\forall\exists}}=r_{\sigma}$ for all $\forall\exists$ sentences $\sigma$, as desired.
\end{proof}

It would be interesting to study $\Th_{\forall\exists}(\cM_{\forall\exists})$ further.  In particular, the following question is natural:

\begin{question}
    Does $\Th_{\forall\exists}(\cM_{\forall\exists})$ coincide with $\Th_{\forall\exists}(L(\bb F_\infty))$ or with $\Th_{\forall\exists}(\prod_{\cU}M_n(\bb C))$?
\end{question}

It would also be interesting to see if there is a non-relative version of $\cal M_{\forall\exists}$, that is, a version that does not require the embeddability assumption.  The issue with adapting the proof of Corollary \ref{maximalAE} to the non-embeddable setting is that it is not clear that given two $\forall\exists$-sentences $\sigma_1$ and $\sigma_2$, there are II$_1$ factors $\cM_1$ and $\cM_2$ which embed into each other's ultrapowers and for which $\sigma_i^{\cal M_i}$ is the maximal value $r_{\sigma_i}$.

\section{Assorted further observations}

\subsection{Free group factors and matrix ultraproducts}

The following question is well-known to experts:

\begin{question}\label{freematrix}
Are there an integer $n$ and an ultrafilter $\cU$ such that $L(\bb F_n)$ has the same theory as the matrix ultraproduct $\prod_{\cal U} M_k(\bb C)$? 
\end{question}

One motivation for the above question is the fact that semi-circular variables arise as limits of random matrices and the free group factor $L(\bb F_n)$ is generated by $n$ free semi-circular elements \cite{VoiCirc1990}. This latter fact says we can approximate the value of quantifier-free sentences in $L(\bb F_n)$ using the asymptotics of random matrices, and the question asks if this approximation can be extended to arbitrary sentences.

In this section, we make a few small observations about this question.

\begin{prop}
For any nonprincipal ultrafilter $\cal U$, we have $\prod_{\cal U}(M_k(\bb C)*L(\bb Z))\cong L(\bb F_2)^{\cal U}$.
\end{prop}

\begin{proof}
By Facts \ref{interpfact} and \ref{freedimfact}, we have that $M_k(\bb C)* L(\bb Z)\cong L(\bb F_{2-\frac{1}{k^2}}).$  
Now apply Proposition \ref{continuity}.
\end{proof}

\begin{lem}
For any nonprincipal ultrafilter $\cal U$, we have $$Th_{\forall\exists} \left(\prod_{\cal U}M_k(\bb C) \right) \leq Th_{\forall\exists}\left( \big(\prod_{\cal U}M_k(\bb C) \big) *L(\bb Z) \right).$$
\end{lem}

\begin{proof}
Fix a separable elementary substructure $\cM$ of $\prod_{\cal U}M_k(\bb C)$.  Then Proposition \ref{popa} implies that there is an existential embedding $\cM \hookrightarrow_1 \cM*L(\bb Z)$.  Lemma \ref{ecfree} then implies that there is an existential embedding $\cM *L(\bb Z)\hookrightarrow_1 (\prod_{\cal U}M_k(\bb C))*L(\bb Z)$, whence there is an existential embedding $\cM \hookrightarrow_1 (\prod_{\cal U}M_k(\bb C))*L(\bb Z)$.  This finishes the proof.
\end{proof}

% An application of CH at the end of the previous proof would in fact yield an existential embedding $\prod_{\cal U}M_k(\bb C)\hookrightarrow (\prod_{\cal U}M_k(\bb C))*L(\bb Z)$.  

In connection with Question \ref{freematrix}, we ask if we can ``bring $L(\bb Z)$ into the ultraproduct''?

\begin{question}\label{AEmatrixultra}
    Is it true that, for any nonprincipal ultrafilter $\cal U$, we have $$Th_{\forall\exists} \left( \prod_{\cal U}M_k(\bb C) \right) \leq Th_{\forall\exists} \left( \prod_{\cal U} \big( M_k(\bb C)*L(\bb Z) \big) \right) =Th_{\forall\exists}(L(\bb F_2))?$$  In fact, is there an existential embedding $\prod_{\cal U}M_k(\bb C)\hookrightarrow_1 L(\bb F_2)^{\cal U}$?
\end{question}

    For $r\geq 2$, Corollary \ref{cor: fgf-ec-embeddings} gives an existential embedding $L(\F_2)^\cU \hookrightarrow_1 L(\F_r)^\cU$; thus, a positive answer to the second part of Question \ref{AEmatrixultra} yields an existential embedding $\prod_{\cal U}M_k(\bb C)\hookrightarrow_1 L(\bb F_r)^{\cal U}$.  We now show that the same conclusion can be reached for $r\in (1,2)$.

% \textcolor{blue}{It would be great to prove that $\sigma^{L(\bb F_2)}\leq \liminf \sigma^{M_k(\bb C)}$ for any $\forall\exists$-sentence $\sigma$, so that $\prod_{\cal U}M_k(\bb C)$ and $L(\bb F_2)$ have the same $\forall\exists$-theory (which would explain why we are having so much trouble distinguishing them).}

\begin{prop}
Suppose the second part of Question \ref{AEmatrixultra} has a positive answer.  Then for any nonprincipal ultrafilter $\cal U$ and any $r\in (1,2)$, we have an existential embedding $\prod_{\cal U}M_k(\bb C)\hookrightarrow_1 L(\bb F_r)^{\cal U}$.  Consequently, $$\limsup_{k\to \infty} \Th_{\forall \exists}({M_k(\bb C)}) \leq \inf_{r\in (1,\infty)}\Th_{\forall \exists}({L(\bb F_r)}).$$
\end{prop}

\begin{proof}
Fix $r\in (1,2)$ and take $m\in \bb N$ sufficiently large such that $1+\frac{1}{m^2}\leq r$.  Take $j\in \{0,1,\ldots,m-1\}$ such that $k-j$ is divisible by $m$ for $\cal U$-almost all $k\in \bb N$.  For this $j$, there is an existential embedding $\prod_{\cal U}M_{\frac{k-j}{m}}(\bb C)\hookrightarrow_1 L(\bb F_2)^{\cal U}$.  (This is an abuse of notation:  $\frac{k-j}{m}$ is an integer for $\cal U$-many $k$, and how we define the embedding on the other factors is irrelevant.) By tensoring with $M_m(\bb C)$, we then have existential embeddings $\prod_{\cal U}M_{k-j}(\bb C)\hookrightarrow_1 L(\bb F_2)_m^{\cal U}\hookrightarrow_1 L(\bb F_r)^{\cal U}$.  It remains to note that $\prod_{\cal U}M_{k-j}(\bb C)\cong \prod_{\cal U}M_k(\bb C)$.  Indeed, for each $k\in \bb N$, take a projection $p_k\in \cal P(M_{k}(\bb C))$ of trace $\frac{k-j}{k}$.  Note then that $\lim_{\cal U} tr(p_k)=1$, so $(p_k)_{\cal U}=1$ and $\prod_{\cal U}M_k(\bb C)=\prod_{\cal U}p_kM_k(\bb C)p_k\cong \prod_{\cal U}M_{k-j}(\bb C)$. 
\end{proof}

\subsection{The question for reduced group C*-algebras}

The following result of Pimsner and Voiculescu \cite{Pimsner-Voi1982} settled the reduced group C*-algebra version of the free group factor problem:

\begin{fact}
For any $m\geq 2$, $K_1(C^*_r(\bb F_m))\cong \bb Z^m$.  Consequently, for distinct $m,n\geq 2$, we have $C^*_r(\bb F_m)\not\cong C^*_r(\bb F_n)$.
\end{fact}

It is of course natural to ask the following:

\begin{question}\label{C*question}
For distinct $m,n\geq 2$, do we have $C^*_r(\bb F_m)\equiv C^*_r(\bb F_n)$ (in the language of unital C*-algebras)?
\end{question}

As far as these authors can tell, this problem is currently open.  In this section, we make two observations about this problem.  The first observation is that a positive answer to the previous problem implies a positive answer to the first-order free group factor alternative:

\begin{lem}
If $C^*_r(\bb F_m)\equiv C^*_r(\bb F_n)$, then $L(\bb F_m)\equiv L(\bb F_n)$.
\end{lem}

\begin{proof}
As discussed in the proof of \cite[Proposition 7.2.4]{MTofCstar}, if $\Gamma$ is a Powers group, it has the uniform strong Dixmier property and thus the unique trace on $C^*_r(\Gamma)$ is definable in the language of unital C*-algebras, with the same definition working for all Powers groups.  By the results of \cite[page 234]{delaHarpe85}, for each $m\geq 2$, $\bb F_m$ is a Powers group, whence the trace on $C^*_r(\bb F_m)$ is definable, uniformly in $m$.  Since $L(\bb F_m)$ is the von Neumann algebra generated by $C^*_r(\bb F_m)$, by \cite[Proposition 3.5.1]{MTofCstar}, we have that $L(\bb F_m)$ is interpretable in $C^*_r(\bb F_m)$, uniformly in $m$ (since the definition of the trace is uniform over all $m$).  The conclusion of the lemma follows.
\end{proof}

We end this section by showing that a positive answer to a well-known problem in C*-algebra theory would lead to a negative answer to Question \ref{C*question}.  Recall that, for a unital C*-algebra $A$, $U_0(A)$ denotes the connected component of the identity in the unitary group $U(A)$. 

\begin{lem}
If $U_0(C^*_r(\bb F_m))$ is a definable subset of $C^*_r(\bb F_m)$, uniformly in $m$, then Question \ref{C*question} has a negative answer.
\end{lem}

\begin{proof}
By a result of Rieffel \cite{Rieffel1983}, if $A$ has stable rank $1$, then $K_1(A)\cong U(A)/U_0(A)$.  By a result of Dykema, Haagerup, and Rordam \cite{DHR-StableRank1997}, $C^*_r(\bb F_m)$ has stable rank $1$ for all $m\geq 2$.  Thus, the assumptions of the lemma would imply that $K_1(C^*_r(\bb F_m))\cong \bb Z_m$ is interpretable in $C^*_r(\bb F_m)$, uniformly in $m$.  Since $\bb Z_m\not\equiv \bb Z_n$ (as groups) for distinct $m,n\geq 2$, the result follows.
\end{proof}

\begin{remark}
    Recall that every element of $U_0(A)$ can be written as a product of elements of the form $\exp(ih)$ for $h$ a self-adjoint element of $A$.  A well-known question in the C*-algebra literature \cite[Problem 2.10]{Phillips1994}  asks if there is a bound on the \textbf{exponential rank} of $C^*_r(\bb F_m)$, that is, whether there is a bound on the number of elements in such a product for an element of $U_0(C^*_r(\bb F_m))$.  As pointed out to us in private communication by Leonel Robert, the finiteness of the exponential rank of $C^*_r(\bb F_m)$ is equivalent to the finiteness of the \textbf{exponential length} of $C^*_r(\bb F_m)$.  (See \cite{Ringrose92} for the definition of exponential length.  In general, finite exponential length implies finite exponential rank.)  As pointed out in \cite[Subsection 3.12]{MTofCstar}, if the exponential length of $A$ is finite, then $U_0(A)$ is definable.  Consequently, if there is a bound on the exponential rank/length of $C^*_r(\bb F_m)$, uniformly in $m$, then Question \ref{C*question} would have a negative answer.
\end{remark}

\vspace{0.7cm}
Goldbring was partially supported by NSF grant DMS-2054477.

\printbibliography

\end{document}